\documentclass{amsart}
\usepackage{amsmath,amssymb}
\usepackage{graphicx,subfigure}
\usepackage{color}

\newtheorem{theorem}{Theorem}[section]
\newtheorem{proposition}[theorem]{Proposition}
\newtheorem{corollary}[theorem]{Corollary}

\newtheorem{lemma}[theorem]{Lemma}

\theoremstyle{definition}

\newtheorem*{definition*}{Definition}

\theoremstyle{remark}
\newtheorem{remark}[theorem]{Remark}

\numberwithin{equation}{section}

\DeclareMathOperator\sgn{sign}

\newcommand{\al}{\alpha}

\newcommand{\ep}{\varepsilon}

\newcommand{\ka}{\kappa}
\newcommand{\la}{\lambda}
\newcommand{\om}{\omega}

\newcommand{\te}{\theta}
\newcommand{\vp}{\varphi}

\newcommand{\De}{\Delta}

\newcommand{\La}{\Lambda_0}

\newcommand{\Om}{\Omega}

\newcommand{\bB}{\mathbf{B}}

\newcommand{\tp}{\widetilde{p}}

\newcommand{\tF}{\widetilde{F}}

\def\RR{\mathbb{R}}

\def\ZZ{\mathbb{Z}}
\def\PP{\mathbb{P}_0}
\def\TT{\mathbb{T}}

\newcommand{\cC}{{\mathcal C}}

\newcommand{\cF}{{\mathcal F}}
\newcommand{\cG}{{\mathcal G}}
\newcommand{\cH}{{\mathcal H}}

\newcommand{\cR}{{\mathcal R}}

\newcommand{\pd}{\partial}
\newcommand\minus\backslash

\newcommand\lan\langle
\newcommand\ran\rangle

\newcommand{\sign}{\operatorname{sign}}

\DeclareMathOperator\Div{div}

\newcommand\DD{\mathbb D}
\renewcommand\leq\leqslant
\renewcommand\geq\geqslant

\newcommand\Dir{_{\mathrm{Dir}}}

\newlength{\intwidth}

\newcommand\loc{_{\mathrm{loc}}}

\newcommand\dphi{\Phi}
\newcommand\dB{\mathbf B}
\newcommand\dphieBB{\Phi_{\ep,B,\bB}}
\newcommand\dphie{\Phi_{\ep,0,\dB}}
\newcommand\phieB{\phi_{\ep,B}}
\newcommand\phie{\phi_{\ep,0}}
\newcommand\ceB{c_{\ep,B}}
\newcommand\ce{c_{\ep,0}}
\newcommand\cCe{\cC_{\ep,\dB}}
\newcommand\PPe{\mathbb{P}_{\ep B}}

\begin{document}

\title[Stationary
  Euler flows with compact support]{Piecewise smooth stationary
  Euler flows with compact support via overdetermined boundary problems}

\author{Miguel Dom\'{\i}nguez-V\'{a}zquez}
\address{Departamento de Matem\'aticas, Universidade de Santiago de
  Compostela, Spain.}

\email{miguel.dominguez@usc.es}
\author{Alberto Enciso}
 \address{Instituto de Ciencias Matem\'aticas, Consejo Superior de
 	Investigaciones Cient\'ificas, Madrid, Spain.}
\email{aenciso@icmat.es}
\author{Daniel Peralta-Salas}
\address{Instituto de Ciencias Matem\'aticas, Consejo Superior de
 	Investigaciones Cient\'ificas, Madrid, Spain.}
\email{dperalta@icmat.es}

%
%
\begin{abstract}
  We construct new stationary weak solutions of the 3D Euler equation
  with compact support. The solutions, which are piecewise smooth and
  discontinuous across a surface, are axisymmetric with swirl. The
  range of solutions we find is different from, and larger than, the
  family of smooth stationary solutions recently obtained by Gavrilov
  and Constantin--La--Vicol; in particular, these solutions are not
  localizable. A key step in the proof is the construction of
  solutions to an overdetermined elliptic boundary value problem where one
  prescribes both Dirichlet and (nonconstant) Neumann data.
\end{abstract}
\maketitle

\section{Introduction}

In two dimensions, it is easy to construct stationary solutions with compact
support to the incompressible Euler equation
\begin{equation}\label{Euler}
\pd_t u + u\cdot\nabla u+\nabla p=0\,,\qquad \Div u=0\,.
\end{equation}
For instance, an explicit $C^\infty$~stationary Euler flow supported in the unit
disk~$\DD$ is given by the stream function $\psi:= e^{1/(|x|^2-1)}\,
{1}_{\DD}(x)$, which determines the velocity field as
$u=\nabla^\perp \psi$.

The question of whether there are stationary Euler flows of compact
support in three dimensions is much harder, and has attracted much
recent attention. In the important particular case that the stationary
Euler flow is a generalized Beltrami field, Nadirashvili~\cite{N} and
Chae--Constantin~\cite{CC} have shown that no compactly supported
solutions exist, and that in fact there are not even any generalized
Beltrami fields with finite energy. It is also known that axisymmetric
stationary Euler flows of compact support without swirl do not exist~\cite{JX}. In contrast, it has
long been known that there exist $C^{1,\al}$~stationary Euler flows
whose vorticity is compactly supported~\cite{FB}. Weak stationary
Euler flows in~$L^\infty$ of compact support can also be constructed
with the convex integration technique developed in~\cite{CS}.

A major recent breakthrough was Gavrilov's construction of compactly
supported stationary Euler flows in three dimensions~\cite{Gavrilov},
which are axisymmetric and of class~$C^\infty$. More concretely, these solutions are of the form
\begin{equation}\label{uGavrilov}
u= G(p_R)\, u_R\,, \qquad p=\int_0^{p_R}G(q)^2\, dq\,,
\end{equation}
where $(u_R,p_R)$ are certain concrete functions
depending on a positive parameter~$R$ that solve the stationary Euler equation in a toroidal domain, and $G$ is an essentially
arbitrary function of one variable (compactly supported). Slightly more general solutions
can be constructed using the same method.

These solutions have been revisited and put
in a broader context using the perspective of the Grad--Shafranov
equation by Constantin, La and Vicol~\cite{CV}, which allowed them to develop nontrivial applications
to other equations of fluid mechanics as well. As stressed by these authors,
the key property of these solutions is that they are {\em
  localizable}\/, meaning that the pressure is constant along the
stream lines of the flow: $u_R\cdot \nabla p_R=0$ (which in turn implies that $u\cdot \nabla p=0$).

Our objective in this paper is to derive a different approach to the
construction of stationary Euler flows with compact support. The solutions
we construct are very different from those obtained by Gavrilov and
Constantin--La--Vicol: they are piecewise smooth
stationary weak solutions with axial symmetry, and they are not
localizable. Each stationary solution we construct is bounded (with bounded vorticity) and supported on a toroidal
domain with a smooth boundary; the flow is smooth in the
interior of this domain (up to the boundary) and possibly discontinuous across the boundary. As
we will make precise below, this approach yields a wide range of axisymmetric Euler flows of compact support.

Our construction of stationary Euler flows with compact support is
based on showing the existence of nontrivial solutions to a boundary
value problem for an elliptic equation where both Dirichlet and
Neumann data are prescribed. This kind of boundary problems are
usually called {\em overdetermined}\/.

To see how overdetermined boundary problems appear in this context,
let us start by recalling the Grad--Shafranov formulation of the
axisymmetric Euler equation in three dimensions. This consists in writing an axisymmetric solution to the Euler equation in
cylindrical coordinates (in terms of the orthonormal basis $\{e_z,e_r,e_\vp\}$) as
\begin{equation}\label{defu}
  u= \frac1r\left[ \pd_r\psi\, e_z-\pd_z\psi\, e_r +F(\psi)\, e_\vp\right]\,,
\end{equation}
where the function $\psi(r,z)$ satisfies the equation
\begin{equation}\label{elliptic}
L\psi=r^2 H'(\psi) -\frac12 (F^2)'(\psi)
\end{equation}
for some function~$H$. The pressure is then given by
\begin{equation}\label{defp}
p= H(\psi)-\frac1{2r^2}\left[|\nabla \psi|^2 +F(\psi)^2\right]
\end{equation}
and we have set
\begin{equation}\label{defL}
L\psi:=\pd_{rr}\psi+\pd_{zz}\psi - \frac1r\pd_r\psi\,.
\end{equation}
The functions $F$ and $H$ can be picked freely.

The first observation
of this paper, which explains why we are interested in overdetermined
boundary problems in this context, is the following: 

\begin{lemma}\label{L.weak}
Let $\Om$ be a $C^2$~bounded domain whose closure is contained in the half-space
$\{(r,z)\in\RR^2: r>0\}$. Suppose that there is a function~$\psi\in
C^2(\Om)$ satisfying Equation~\eqref{elliptic} in~$\Om$ and the
boundary conditions
\begin{align}
  \psi|_{\pd\Om}=0\,,\label{Dirichlet}\\
  \left.\frac{(\pd_\nu\psi)^2+ F(0)^2}{r^2}\right|_{\pd\Om}=c\,,\label{Neumann}
\end{align}
where $c$ is a constant, and $\nu$ is a unit normal field on $\partial\Omega$. Then the vector field~$u$ defined
by~\eqref{defu} inside~$\Om$ and as $u:=0$ outside~$\Om$ is a weak
solution of the Euler equation, the pressure being given
by~\eqref{defp} inside~$\Om$ and by $p:=H(0)-c/2$ outside~$\Om$.
\end{lemma}

Therefore, the key step of our construction of stationary Euler flows
with compact support is to show the existence of nontrivial solutions
to a (non-standard) overdetermined boundary problem for a certain
semilinear elliptic equation in two variables. While the investigation of overdetermined boundary value problems traces
back to Serrin's seminal paper~\cite{Serrin} in 1971, until very
recently the literature on overdetermined boundary problems was
essentially limited to proofs that solutions need to be radial in
cases that could be handled using the method of moving planes.

However, in two surprising papers, Pacard and Sicbaldi~\cite{PS:AIF}
and Delay and Sicbaldi~\cite{DS:DCDS} proved the existence of extremal
domains with small volume for the first eigenvalue $\lambda_1$ of the
Laplacian in any compact Riemannian manifold, which guarantees the
existence of solutions to a certain overdetermined problem
for the linear elliptic operator $\De+\la_1$ in a domain with both zero Dirichlet data and constant
Neumann data. Very recently we managed to show the existence of
nontrivial solutions, still with the same Dirichlet and Neumann data, for
fairly general semilinear elliptic equations of second order with possibly nonconstant
coefficients~\cite{overdet}. Our strategy here is to start by tweaking the
proof of this result to accomodate to the non-standard boundary
condition we must impose.

The main difficulty to solve Equation~\eqref{elliptic} with the
overdetermined boundary conditions~\eqref{Dirichlet}
and~\eqref{Neumann} is that the Neumann data depend on the point of
the boundary $\partial\Om$, a situation that was not considered in our
paper~\cite{overdet}. The technique used there is based on a
variational technique that relates the existence of overdetermined
solutions with the critical points of certain functional, a strategy
that is successful only for constant boundary data. Roughly speaking,
the gist of the argument is that the overdetermined condition with
constant data is connected with the local extrema for a natural energy
functional, restricted to a specific class of functions labeled by
points in the physical space. This fact
ultimately permits to derive the existence of solutions from the fact
that a continuous function attains its maximum on a compact
manifold. We do not know of an analog of this fact for the nonconstant
boundary conditions considered in this paper. Instead, we rely on a new,
hands-on approach to the problem that results in a flexible, very explicit existence theorem:

\begin{theorem}\label{T.main}
Consider any
functions $\tF,H\in C^s((-1,0])$ satisfying
\[
\tF(0)=0\,,\qquad \tF'(0)<0\,,\qquad H'(0)>0\,,
\]
where $s>2$ is not an integer. Then the following statements hold:
\begin{enumerate}
\item For each small enough
  $\ep>0$ and any
   large enough~$R>0$,
  there exists a nontrivial, piecewise~$C^s$, axisymmetric stationary
  Euler flow of compact support~$u$ of the form described in Lemma~\ref{L.weak} for
  a suitable $C^{s+1}$~planar domain~$\Om_{R,\ep}$.
  \item The domain~$\Om_{R,\ep}$ is a small
  deformation of a disk of radius~$\ep$ centered at the point $(R,0)$ of the $(r,z)$-plane.
\item The functions that define the solution are
    \begin{equation}\label{FH}
F(\psi) := [\ep^2 F_R+ \tF(\psi)]^{1/2}
\end{equation}
and $H(\psi)$, where $F_R$ is a positive constant depending
on~$R$.
\item The function $\psi$, which is of class $C^{s+1}$ in
  $\Om_{R,\ep}$ up to the boundary, is approximately radial.
Moreover, $\ep^2F_R+\tF\circ\psi>0$, $F\circ \psi>0$ and $H\circ\psi$ are of class $C^s$
in $\Om_{R,\ep}$.
    \end{enumerate}
  \end{theorem}

  \begin{remark}
    In fact, the value of the constants and the structure of the solutions is completely explicit:
    \begin{enumerate}
    \item $R$ must be larger than $[-\frac32\tF'(0)/H'(0)]^{1/2}$.
      \item The boundary of~$\pd\Om_{R,\ep}$ is the curve defined by
        an equation of the form $z^2+(r-R)^2-\ep^2 = O(\ep^3)$.
        \item The constant $F_R$ is
\[
F_R:=\frac{1}{16}\left[H'(0)R^2-\frac12\tF'(0)\right]\left[H'(0)R^2+\frac32\tF'(0)\right]>0\,,
\]
and the constant $c$ in the Neumann condition is of the form
\[
c=\frac{\ep^2}{16R^2}\left[H'(0)R^2-\frac12\tF'(0)\right]\left[5H'(0)R^2-\frac12\tF'(0)\right]+O(\ep^3)\,.
\]
\item The function~$\psi$ is of the form
  $$
\psi=\frac14 \left[H'(0)R^2-\frac12\tF'(0)\right]\left[(r-R)^2+z^2-\ep^2\right]+O(\ep^3)\,.
$$
\item The vorticity,
  \[
\om= \frac{F'(\psi)}r\left( \pd_r\psi\, e_z-\pd_z\psi\, e_r\right) +\left[-rH'(\psi)+\frac{(F^2)'(\psi)}{2r}\right]\, e_\vp\,,
\]
is also of class $C^{s-1}$ up to the boundary.
    \end{enumerate}
  \end{remark}

Several comments are in order. First, let us emphasize that the
solutions we construct are piecewise smooth but discontinuous
across a smooth surface; hence, from the point of view of their
regularity, they stand between the smooth solutions
of~\cite{Gavrilov,CV} and the rough weak solutions
of~\cite{CS}. Concerning the flexibility of the construction, it is
apparent that the range of solutions we obtain is much larger than
that of~\cite{Gavrilov,CV}. Indeed, our solutions are not localizable in general and
are of class $C^s$ in the toroidal domain of~$\RR^3$ defined by the
planar domain~$\Om$. In contrast, the
basic vector field~$u_R$ that appears in~\eqref{uGavrilov} is not
defined at an inside point (``the origin''), so the function~$G$ in that equation is chosen
so that it vanishes to infinite order there; in fact, the supports of the solutions
constructed so far are toroidal shells instead of solid tori. Actually,
the functions~$H$ and~$F$ in~\cite{Gavrilov,CV} defining the vector
field~$u_R$ are precisely $H(\psi)=a\psi$ and a certain function of
the form $F(\psi)=Rb\sqrt\psi\,[1+O(\psi)]$, where again the
positive constants~$a$ and~$b$ are fixed.

It should be noted that the philosophy that underlies the proof of Theorem~\ref{T.main}
has, in fact, a wider range of applicability. To illustrate this fact,
we include in Section~\ref{S.final} a brief discussion on the
existence of compactly supported solutions for a class of functions~$F$
and~$H$ different from that considered above.

The paper is organized as follows. In Section~\ref{S.weak} we will
prove Lemma~\ref{L.weak} as a corollary of a more general result about
piecewise smooth weak solutions to the stationary Euler equation. The
rest of the paper is devoted to the proof of Theorem~\ref{T.main}. In
Section~\ref{S.Dirichlet} we construct solutions to the Dirichlet
problem for the Equation~\eqref{elliptic}, and subsequently compute
their asymptotic expansion in the parameter $\ep$
(Section~\ref{S.B=0}). The way these solutions change when the domain
is perturbed a little is discussed in Section~\ref{S.variations}. This
is key for the proof that there exists a domain which is $\ep^2$-close
to a disk of radius $\ep$ where the Dirichlet solution also
satisfies the additional Neumann condition~\eqref{Neumann}. This
result, which we prove in Section~\ref{S.Neumann}, allows us to
complete the proof of Theorem~\ref{T.main}. To conclude, in
Section~\ref{S.final} we briefly discuss the existence of compactly
supported solutions defined by functions~$F$ and~$H$ different from
those considered in Theorem~\ref{T.main}.

\section{Stationary Euler flows via an overdetermined boundary problem}
\label{S.weak}

In this short section we prove that if one has a stationary Euler flow
on a bounded domain~$\Om$ which is tangent to the boundary and whose pressure
is constant on~$\pd\Om$, then one can trivially extend it to a stationary Euler
flow on the whole space with compact support. In the particular case
when the initial flow is axisymmetric, and hence described by the Grad--Shafranov
formulation (Equation~\eqref{defu}), this reduces to
Lemma~\ref{L.weak} stated in the Introduction.

Let us start by recalling the definition of a weak stationary Euler
flow. We say that a pair $(u,p)$ of class, say, $L^2\loc(\RR^3)$ is a
{\em weak solution to the stationary Euler equation}\/ if
\[
\int_{\RR^3} \left[(u\otimes u)\cdot \nabla w+ p\Div w\right]\, dx=0\quad
\text{and}\quad \int_{\RR^3}u\cdot \nabla\phi\,dx=0
\]
for any vector field $w\in C^\infty_c(\RR^3)$ and any scalar function
$\phi\in C^\infty_c(\RR^3)$. Of course, if $u$ and~$p$ are smooth enough, this
is equivalent to saying that they satisfy Equation~\eqref{Euler}
in~$\RR^3$.

\begin{lemma}\label{L.Euler}
Given a bounded domain~$\Om$ in~$\RR^3$ with a $C^2$~boundary, suppose that
$v\in C^1(\Om)\cap L^2(\Om)$ is a solution to the stationary
Euler equation in~$\Om$ with pressure $\tp\in C^1(\Om)\cap L^1(\Om)$. Assume that
$\nu\cdot v|_{\pd\Om}=0$ and $\tp|_{\pd\Om}=c$, where $c$ is a constant. Then
\[
u(x):=\begin{cases}
v(x) &\text{ if }\; x\in\Om\,,\\
0 &\text{ if }\; x\not\in\Om\,,
\end{cases}
\]
is a weak solution of the stationary Euler equation on~$\RR^3$
with pressure
\[
p(x):=\begin{cases}
\tp(x) &\text{ if }\; x\in\Om\,,\\
c &\text{ if }\; x\not\in\Om\,.
\end{cases}
\]
\end{lemma}

\begin{proof}
We start by noticing that, for all $\phi\in C^\infty_c(\RR^3)$,
\begin{align*}
\int_{\RR^3}u\cdot \nabla \phi\, dx=\int_\Om v\cdot \nabla \phi\,
                                       dx = -\int_\Om \phi\,\Div v\, dx + \int_{\pd \Om} \phi\, \nu\cdot v\, dS=0\,,
\end{align*}
where we have used that $\Div v=0$ in~$\Om$ and $\nu \cdot v=0$ on~$\pd\Om$. Hence
$\Div u=0$ in the sense of distributions.

Let us now take an arbitrary vector field $w\in
C^\infty_c(\RR^3)$. As $\int_{\RR^3} \Div w\,dx =0$, we can write
\begin{align*}
  \int_{\RR^3} [(u&\otimes u)\cdot \nabla w+ p\Div w]\,
                                                 dx
  =
  \int_{\Om} (u\otimes u)\cdot \nabla w\, dx + \int_{\RR^3}(p-c)\, \Div w\,
  dx\\
  &=\int_\Om \left[(v\otimes v)\cdot \nabla w- w\cdot \nabla \tp\right]\, dx
  + \int_{\pd\Om} (\tp-c) \, w\cdot
  \nu\, dS  \\
  &=-\int_{\Om} \left[\Div(v\otimes v) + \nabla \tp\right]\cdot w\, dx + \int_{\pd\Om} \left[ (v\cdot w)(v\cdot \nu) + (\tp-c) \, w\cdot
  \nu\right]\, dS\,.
\end{align*}
The volume integral is zero because $v$ satisfies the stationary Euler
equation in~$\Om$. Since $v\cdot \nu=\tp-c=0$ on~$\pd\Om$, the surface integral vanishes too. It then follows
that $u$ is a weak solution of the Euler equation in $\RR^3$, as claimed.
\end{proof}

\begin{remark}
  With $u$ given by~\eqref{defu}, and using the notation in Lemma~\ref{L.weak}, the conditions that $u$ be
  tangent to the axisymmetric domain defined by~$\Om$ and that the pressure~\eqref{defp}
  be constant on the boundary amount to saying that $\psi$ takes a
  constant value~$c_0$
  on~$\pd\Om$ and that $[(\pd_\nu\psi)^2+ F(c_0)^2]/r^2$ is also
  constant on~$\pd\Om$. Setting $c_0:=0$ without loss of generality, Lemma~\ref{L.Euler}
  results in the statement of Lemma~\ref{L.weak}.
\end{remark}

\section{Solutions to the Dirichlet boundary problem}
\label{S.Dirichlet}

Let us take any $R>0$ that will be fixed during the whole construction and
introduce suitably translated and rescaled variables $(x,y)\in\RR^2$
as
\[
r=:R+\ep x\,,\qquad z=: \ep y\,.
\]
Throughout, $\ep>0$ denotes a small parameter. We will also consider the polar
coordinates $(\rho,\te)\in (0,\infty)\times \TT$, where
$\TT:=\RR/2\pi\ZZ$, that are defined in terms of $(x,y)$ through the formulas
\[
x= \rho\cos\te\,,\qquad y=\rho\sin\te\,.
\]
In these variables, the Grad--Shafranov
equation~\eqref{elliptic} reads as
\begin{equation}\label{GS2}
\De\psi -\frac\ep{R+\ep x}\pd_x \psi= \ep^2(R+\ep x)^2 H'(\psi) -\frac12\ep^2 (F^2)'(\psi)\,,
\end{equation}
where $F$ and~$H$ are of the form described in Theorem~\ref{T.main}.

We look for solutions to Equation~\eqref{GS2} in domains of the form
\[
\Om_{\ep B}:=\{ \rho<1+\ep B(\te)\}
\]
for some function $B\in C^{s+1}(\TT)$; notice that $\Om_{\ep B}$ only
depends on the product $\ep B$ and not on~$\ep$ and~$B$ separately. To keep track
of the size of~$\psi$, we will set
\[
\psi=:\ep^{2}\phi\,.
\]
For $\ep\neq0$, Equation~\eqref{GS2} can be written in
terms of~$\phi$ as
\begin{equation}\label{eq}
\De\phi -\frac\ep{R+\ep x}\pd_x \phi= aR^2+b + 2aR\ep x + \cR(x,\phi)\,,
\end{equation}
where we have defined the positive constants
$$a:= H'(0) \qquad b:=-\frac12\tF'(0)\,,$$
and where
\[
\cR(x,\phi):= \ep^2 ax^2 +  (R+\ep x)^2 H_1'(\ep^2\phi) -\frac12  \tF_1'(\ep^2\phi)\,.
\]
Here
\[
H_1(\psi):= H(\psi)-a\psi-H(0)\,,\qquad \tF_1(\psi):= \tF(\psi)+2b\psi
\]
are functions that vanish to second order at $\psi=0$, so we have the obvious bound
\[
\sup_{\|\phi\|_{C^{s-1}(\Om)}<C}\|\cR\|_{C^{s-1}(\Om)}\leq C'\ep^2\,.
\]
The Dirichlet boundary condition $\psi=0$ on~$\pd\Om_{\ep B}$ can then be
rewritten in terms of $\phi(\rho,\te)$ as
\begin{equation}\label{DBC}
\phi(1+\ep B(\te),\te)=0\,.
\end{equation}

Henceforth we will say that a function $f(\rho,\te)$ is {\em
  even}\/ if
\[
  f(\rho,\te)=f(\rho,-\te)\,,
\]
and similarly for a function $g(\te)$. Equivalently, a function is
even if it is invariant under the reflection $y\mapsto -y$.

Since the function
$$\phi_0:= \frac{(aR^2+b)(\rho^2-1)}{4}$$
satisfies Equation~\eqref{eq} and the Dirichlet condition~\eqref{DBC} when
$\ep=0$, it is straightforward to show that there are solutions to
the problem for small~$\ep$ using the implicit function theorem in Banach spaces:

\begin{proposition}\label{P.existence}
For any small enough~$\ep $ and any function~$B$ with
$\|B\|_{C^{s+1}}<1$ there is a unique function~$\phi=\phieB$ in a
small neighborhood of~$\phi_0$ in~$C^{s+1}(\Om_{\ep B})$ that satisfies
Equation~\eqref{eq} and the Dirichlet boundary
condition~\eqref{DBC}. Furthermore, $\phieB<0$ in~$\Om_{\ep B}$ and $\phieB$ is even if~$B$ is.
\end{proposition}

\begin{proof}
For small enough~$\ep\neq0$ and $\|B\|_{C^{s+1}}<1$, let
$\chi_{\ep B}:\DD\to \Om_{\ep B}$ be the diffeomorphism defined in polar coordinates as
\[
(\rho,\te)\mapsto \big([1+\ep \chi(\rho)\, B(\te)]\rho,\te\big)\,,
\]
where $\chi(\rho)$ is a smooth cutoff function that is zero for $\rho<1/4$ and $1$
for $\rho>1/2$, and where $\DD:=\{\rho<1\}$ is the unit disk . Then one can define a map
\[
\cH: (-\ep_0,\ep_0)\times C^{s+1}\Dir(\DD)\to C^{s-1}(\DD)
\]
as
\begin{multline*}
\cH (\ep,\bar\phi):=\Big[\De(\bar\phi\circ \chi_{\ep B}^{-1})  -\frac\ep{R+\ep
  x}\pd_x (\bar\phi\circ \chi_{\ep B}^{-1})\\
-[ aR^2+b + 2aR\ep x +
\cR(x,\bar\phi\circ \chi_{\ep B}^{-1})\Big]\circ \chi_{\ep B}\,.
\end{multline*}
Here
\[
  C^{s+1}\Dir(\DD):= \{\bar\phi\in C^{s+1}(\DD): \bar\phi|_{\pd\DD}=0\}\,.
\]
Notice that $\phi:=\bar\phi\circ \chi_{\ep B}^{-1}\in C^{s+1}(\Om_{\ep B})$ and $\phi=0$ on $\pd\Om_{\ep B}$. It is obvious that if $\overline \phi$ solves $\cH(\ep,\bar\phi)=0$, then $\phi$ is a solution to the Dirichlet problem~\eqref{eq}--\eqref{DBC}.

Since $\cH (0,\phi_0)=0$ and the Fr\'echet derivative
\[
D_{\bar\phi} \cH (0,\phi_0)=\De
\]
is an invertible mapping $C^{s+1}\Dir(\DD)\to C^{s-1}(\DD)$, it
follows from the implicit function theorem that for any small enough
$\ep$ and~$B$ there is a unique solution~$\bar\phi_{\ep,B}$ to the equation
$\cH (\ep,\bar\phi_{\ep,B})=0$ in a neighborhood of~$\phi_0$. As $B$ only appears
in the problem through the product $\ep B$, this is equivalent
to the first part of the statement. Also, this uniqueness property immediately implies that
$\phi_{\ep,B}:=\bar\phi_{\ep,B}\circ \chi_{\ep B}^{-1}$ is even when~$B$ is. The property that $\phi_{\ep,B}<0$ in $\Om_{\ep B}$ follows from Equation~\eqref{eq}, i.e.,
$$
\Delta \phi_{\ep,B}=aR^2+b+O(\ep)>0\,,
$$
via the maximum principle.
\end{proof}

\section{Analysis of the solution}
\label{S.B=0}

In the following proposition we compute an asymptotic expansion for
the function~$\phieB$ for small~$\ep$. The constants $A_0$ and~$A_1$
appearing in this expansion, which depend on $R$ but not on~$\ep$, will
play a major role in the rest of the paper.

Throughout, we will denote
by~$\PPe$ the Poisson integral operator of the domain~$\Om_{\ep B}$ in
the coordinates $(\rho,\te)$. That is, $v(\rho,\te):= \PPe f$ denotes the only
harmonic function in~$\Om_{\ep B}$ satisfying the boundary condition
\[
v(1+\ep B(\te),\te)=f(\te)\,.
\]
Note that $\PPe$ only depends on the product $\ep B$, not on~$\ep$
and~$B$ separately. It is standard that $\PPe$ defines a map
$C^{s+1}(\TT)\to C^{s+1}(\Om_{\ep B})$. A convenient explicit formula for the
Poisson operator in the case of the disk is
\[
\PP f(\rho,\te):= \sum_{n\in\ZZ} f_n\,\rho^{|n|}\, e^{in\te}\quad
\text{if} \quad  f(\te)= \sum_{n\in\ZZ} f_n\, e^{in\te}\,.
\]

\begin{proposition}\label{P.phi0}
  For small enough~$\ep$, the function~$\phieB$ has the asymptotic form
  \[
\phieB= A_0(\rho^2-1) + \ep [A_1 (\rho^3-\rho) \cos\te-2A_0\, \PPe B] + O(\ep^2)\,,
\]
with the constants
\[
A_0:=\frac{aR^2+b}4\,,\qquad A_1:= \frac{5aR^2+b}{16R}\,.
\]
\end{proposition}

\begin{proof}
  As it is clear that $\phi_0:=
  A_0(\rho^2-1)$ is the solution to the boundary value
  problem under consideration when $\ep=0$, let us assume $\ep$ is
  nonzero. Note that the equation
  for~$\phi=\phieB$ is of the form
  \[
\De\phi -\frac \ep R\pd_x\phi- (aR^2+b+2aR\ep x)=O(\ep^2)\,,\qquad
\phi(1+\ep B(\te),\te)=0\,.
  \]
  One can then set $\phi_1:=(\phi-\phi_0)/\ep$ and arrive at the
  equation
  \begin{equation*}
\De \phi_1 = 8A_1 x+O(\ep)\,,\qquad \phi_1(1+\ep B(\te),\te) =
-2A_0B(\te)+ O(\ep)\,.
\end{equation*}
A short computation then shows that $h:= \phi_1-\frac43A_1x^3$
  satisfies
  \[
\De h=O(\ep)\,,\qquad h(1+\ep B(\te),\te)=-2A_0B(\te) -\frac43A_1\cos^3\te+O(\ep)\,.
  \]
  Hence
  \begin{align*}
    h&=-2A_0\PPe B-\frac43A_1\PPe(\cos^3\te)+O(\ep)\\
    &=-2A_0\PPe B-\frac43A_1\PP(\cos^3\te)+O(\ep)\,.
  \end{align*}
  As $\cos^3\te= \frac14\cos 3\te
  + \frac34\cos\te$, we then have
  \begin{align*}
    \phi_1&= -2A_0\PPe B+\frac{4A_1}3\left[\rho^3\cos^3\te - \PP(\cos^3\te)\right]+O(\ep)\\
    &=-2A_0\PPe B+ \frac{A_1}3\left[\rho^3(\cos 3\te + 3\cos\te) - (\rho^3\cos3\te
      + 3\rho\cos\te)\right]+O(\ep)\\
    &= -2A_0\PPe B+ A_1(\rho^3-\rho)\cos\te +O(\ep)\,.
  \end{align*}
This is the desired
  expression for~$\phi$.
\end{proof}

\begin{remark}
  Note that we cannot replace $\PPe$ by~$\PP$ in the
  formula presented in Proposition~\ref{P.phi0} because, generally, $\PP B$ would only
  be defined on the unit disk, not on the possibly larger domain
  $\{\rho<1+\ep B(\te)\}$.
\end{remark}

For future reference, we record some formulas that stem from
Proposition~\ref{P.phi0} and will be useful later on:
  \begin{subequations}\label{formulaphi1}
\begin{align}
\nabla \phieB &= [2A_0\rho +\ep A_1(3\rho^2-1)\cos\te]\,
               e_\rho-2A_0\ep\,\nabla\PPe B \notag\\
  &\qquad \qquad \qquad \qquad \qquad \qquad\qquad+ \ep
              A_1(\rho^3-\rho)\,\nabla\cos\te + O(\ep^2)\,,\\
 |\nabla \phieB|^2 &= 4A_0^2\rho^2+ 4\ep A_0 [
                    A_1(3\rho^3-\rho)\cos\te -2A_0\rho\, \pd_\rho\PPe B]+
                   O(\ep^2)\,.
\end{align}
Here $e_\rho:=(x/\rho,y/\rho)$ is the unit vector field in the radial
direction and we have used that $e_\rho\cdot\nabla \cos\te=0$.

Eventually we will need to evaluate the above formulas on the boundary
of the domain, that is, at $\rho=1+\ep
B(\te)$. In this direction, recall that the Dirichlet--Neumann map of the
disk, defined as
\[
\La f(\te)= \pd_\rho\PP f(1,\te)\,,
\]
is the operator $C^{s+1}(\TT)\to C^s(\TT)$ given by
\[
\La f(\te):= \sum_{n\in\ZZ} f_n\,|n|\, e^{in\te}\quad
\text{if} \quad  f(\te)= \sum_{n\in\ZZ} f_n\, e^{in\te}\,.
\]
Also, note that the Dirichlet--Neumann map of the domain $\Om_{\ep B}$ is
an elliptic pseudodifferential operator of first order of the form
\begin{equation}
\Lambda_{\ep B}f:=\pd_\rho\PPe f|_{\rho={1+\ep B}}= \La f+O(\ep)\,,
\end{equation}
where the above notation can be taken to mean that the $C^s$ norm
of the error is bounded by $C\ep\|f\|_{C^{s+1}}$.
\end{subequations}

\section{Computing the variations with respect to the domain}
\label{S.variations}

Our next objective is to compute how $\phi$ changes as we change the
domain by perturbing the function~$B$. More precisely, we aim to
compute the derivative of $\phi$ with respect to~$B$, which we will
denote as
\[
\dphi_{\ep,B,\dB}:= \frac{\pd}{\pd t}\bigg|_{t=0}\phi_{\ep,B+ t\dB}\,,
\]
where $\dB(\te)$ is a function defined on $\TT$. In this section,
we are mainly interested in the derivative at $B=0$, $\dphie$.

In the statement of the next proposition, we will need the operator
$T:C^{s+1}(\TT)\to C^{s+1}(\DD)$ defined as
\[
Tf(\rho,\te):= \sum_{n\in\ZZ\backslash\{0\}} f_n\, (\rho^{|n|-1}-\rho^{|n|+1})\,
e^{i[n-\sgn(n)]\te}
\]
for $f(\te)=\sum_{n\in\ZZ} f_n\, e^{in\te}$. Here and in what follows,
$\sgn(n):=n/|n|$ is the sign of the nonzero integer~$n$.

\begin{proposition}\label{P.dphi}
  For $\dB\in C^{s+1}(\TT)$ and any small enough $\ep$,
  \[
\dphie= -2\ep A_0\,\PP \dB +\ep^2 \left[ \frac{A_0}{2R} \, T\dB
  -2A_1\, \PP(\cos\te\, \dB)\right] + O(\ep^3)\,.
  \]
\end{proposition}

\begin{proof}
Differentiating Equation~\eqref{eq} with respect to~$B$ at $B=0$, we
obtain that $\dphi\equiv \dphie$ satisfies the equation
\begin{equation}\label{Lapdp}
\De\dphi-\frac\ep{R+\ep x}\pd_x\dphi = \ep^2\Big[(R+\ep x)^2H_1''(\ep^2\phi_{\ep,0})-\frac{\ep^2}{2}\tF_1''(\ep^2\phi_{\ep,0})\Big]\, \dphi
\end{equation}
in~$\DD$. Likewise, differentiating the boundary condition~\eqref{DBC}
we obtain that
\[
\dphi(1,\te)= -\ep \pd_\rho\phie (1,\te)\, \dB(\te)\,.
\]
In view of the asymptotics for $\phie$ computed in
Proposition~\ref{P.phi0}, this boundary condition can be rewritten as
\[
\dphi(1,\te)= -2\ep A_0 \dB(\te) -2 \ep^2A_1  \dB(\te)\, \cos\te + O(\ep^3)\,,
\]
and $\dphi$ has the expression
\[
\dphi=-2\ep A_0 \PP\dB+O(\ep^2)\,.
\]
Equation~\eqref{Lapdp} is then of the form
\[
\De\dphi-\frac\ep{R}\pd_x\dphi =O(\ep^2)\dphi + O(\ep^2)\pd_x\dphi=O(\ep^3)\,.
\]

Assuming that $\ep\neq0$ (since otherwise $\Phi=0$), let us set
\[
\Phi_1:= (\Phi-\ep\Phi_0)/\ep^2\,,
\]
with $\Phi_0:= -2A_0\PP\dB$. A short calculation shows that $\Phi_1$
must solve the equation
\[
\De \Phi_1=\frac 1R\pd_x\Phi_0+ O(\ep)
\]
with the boundary condition
\[
\Phi_1(1,\te)=-2A_1\dB(\te) \cos\te + O(\ep)\,.
\]

Since
\[
  \pd_x= \cos\te\,\pd_\rho - \frac1\rho\sin\te\,\pd_\te=\frac12
    e^{i\te} \left(\pd_\rho +\frac i{\rho}\pd_\te\right) +
    \frac12e^{-i\te}\left(\pd_\rho -\frac i\rho \pd_\te\right)\,,
\]
one readily finds that
\[
\pd_x\Phi_0 = -2A_0 \sum_{n\in\ZZ\backslash\{0\}} |n|\dB_n \,\rho^{|n|-1} e^{i[n-\sign(n)]\te}\,,
\]
if $\dB = \sum_{n\in\ZZ} \dB_n\, e^{in\te}$. Let us now note that if
we set
\[
\Phi_2:= -\frac{A_0}{2R}\sum_{n\in\ZZ\backslash\{0\}} \dB_n\rho^{|n|+1}e^{i[n-\sgn(n)]\te}\,,
\]
then $\De \Phi_2=\frac{\pd_x\Phi_0}{R}$. Consequently, the function $\Phi_3:=
\Phi_1-\Phi_2$ satisfies the equation
\[
\De \Phi_3=O(\ep)
\]
and the boundary condition
\begin{align*}
  \Phi_3(1,\te)&= -2A_1\dB(\te) \cos\te-\Phi_2(1,\te)+O(\ep)\\
  &=
-2A_1\dB(\te) \cos\te + \frac{A_0}{2R}\sum_{n\in\ZZ\backslash\{0\}}
    \dB_ne^{i[n-\sgn(n)]\te} + O(\ep)\,.
\end{align*}
This shows that
\begin{align*}
\Phi_3&= -2A_1\PP(\dB \cos\te) + \frac{A_0}{2R}\sum_{n\in\ZZ\backslash\{0\}}
        \dB_n\PP(e^{i[n-\sgn(n)]\te})+ O(\ep)\\
  &= -2A_1\PP(\dB \cos\te) + \frac{A_0}{2R}\sum_{n\in\ZZ\backslash\{0\}}
        \dB_n\rho^{|n|-1} e^{i[n-\sgn(n)]\te}+ O(\ep)\,,
\end{align*}
which results in
\begin{align*}
\Phi_1&= \Phi_2+\Phi_3= \frac{A_0}{2R} \, T\dB
  -2A_1\, \PP(\cos\te\, \dB) + O(\ep)\,,
\end{align*}
as claimed.
\end{proof}

As a consequence of Proposition~\ref{P.dphi}, we record that
\begin{align}\label{formulas2}
\pd_\rho\dphie|_{\rho=1}  = -2\ep A_0\, \La\dB - 2\ep^2\, \left[\frac{A_0}R\,
                 T'\dB+ A_1\,
                 \La(\cos\te\, \dB)\right] + O(\ep^3)\,,
\end{align}
where $T':C^{s+1}(\TT)\to C^{s+1}(\TT)$ is the operator defined as
\begin{equation}\label{T'}
T' f(\te)  :=  \frac12\sum_{n\in\ZZ\backslash\{0\}}f_n\, e^{i[n-\sgn(n)]\te}
\end{equation}
for $f(\te)= \sum_{n\in\ZZ} f_n\, e^{in\te}$.

\section{Analysis of the Neumann condition and conclusion of the proof}
\label{S.Neumann}

Let us now set
\begin{equation}\label{cF}
\cF(\ep,B)(\te):= |\nabla\phieB(1+\ep B(\te),\te)|^2  - \ceB\,
[R+\ep(1+\ep B(\te))\cos\te]^2\,,
\end{equation}
where the constant~$\ceB$ will be defined later. In view of the
definition of the function~$F$ (Equation~\eqref{FH}), one should
notice that the Neumann condition~\eqref{Neumann} holds with a
constant $c=\ep^2 \ceB$ if and only if $\cF(\ep,B)+F_R$ is the zero function.

Next we pick the constant~$\ceB$ so that $\cF(\ep,B)$ is
$L^2$-orthogonal to $\cos\te$. The reason for which we do so will be
clear later. This amounts to
setting
\begin{equation}\label{ceB}
\ceB:= \frac{\int_0^{2\pi} |\nabla\phieB(1+\ep B(\te),\te)|^2\cos\te\,
  d\te}{\int_0^{2\pi} [R+\ep(1+\ep B(\te))\cos\te]^2\cos\te\, d\te}\,.
\end{equation}

The following result guarantees that this choice of~$\ceB$ makes sense
for all small enough~$\ep$, including $\ep=0$, and shows that
$\cF(0,B)$ is in fact the constant
\[
\ka :=4A_0(A_0 -A_1R)\,,
\]
which depends on~$R$ but not on~$B$. In what follows, we
employ the notation
\[
\langle f,g\rangle:= \int_0^{2\pi} f(\te)\, g(\te)\, d\te
\]
for the $L^2$~product on~$\TT$. 

\begin{proposition}\label{P.ceB}
  For small enough~$\ep$ and any~$B$,
  \begin{align*}
\ceB &= \frac{4A_0A_1}R + O(\ep)\,,\\
\cF(\ep,B)&=\ka +O(\ep)\,,\\
\cF(\ep,0)&=\ka +O(\ep^2)\,.
  \end{align*}
\end{proposition}

\begin{proof}
  Let us assume that $\ep\neq0$. In view of Equation~\eqref{ceB}, let
  us write $\ceB= c_1/c_2$, with
  \begin{subequations}\label{cj}
  \begin{align}
c_1&:= \int_0^{2\pi} |\nabla\phieB(1+\ep B(\te),\te)|^2\cos\te\,
     d\te\,,\\
    c_2&:= \int_0^{2\pi} [R+\ep(1+\ep B(\te))\cos\te]^2\cos\te\, d\te\,.
  \end{align}
\end{subequations}
It follows from the formula for $|\nabla\phieB|^2$ derived in~\eqref{formulaphi1} that
  \begin{align}
c_1 &= \int_0^{2\pi} \left[ 4A_0^2 + 8A_0\ep\big( A_0(B - \La B) +
      A_1\cos\te\big)\right] \cos\te\, d\te+O(\ep^2)\notag\\
    &= 8\ep A_0^2 \langle B-\La B, \cos\te\rangle
      +8\ep A_0A_1\int_0^{2\pi}\cos^2\te\, d\te+O(\ep^2) \notag\\
    &= 8\pi\ep A_0A_1+O(\ep^2)\,,\label{c1}
  \end{align}
  where we have used that
  \[
\langle B-\La B, \cos\te\rangle = \langle B, (1-\La)\cos\te\rangle=0
  \]
  for any~$B$ because $\La$ is self-adjoint and $\La(\cos\te)= \cos\te$.

  The computation of $c_2$ is straightforward:
  \begin{equation}\label{c2}
c_2= \int_0^{2\pi}[R^2+2\ep R\cos\te]\cos\te\, d\te + O(\ep^2)= 2\pi
\ep R+ O(\ep^2)\,.
\end{equation}
This readily implies that $\ceB$ can be defined at $\ep=0$ by
continuity and yields the formula for~$\ceB$ presented in the
statement. Also, the above formulas immediately imply that
\begin{align*}
\cF(\ep,B)&= |\nabla\phieB(1+\ep B(\te),\te)|^2- \ceB[R+\ep(1+\ep
              B(\te))\cos\te]^2\\
  &= 4A_0(A_0 -A_1R)+ O(\ep)\,,
\end{align*}
as claimed.

To prove that $\cF(\ep,0)=\ka +O(\ep^2)$, it is convenient to define the ($R$-dependent) constant
\begin{align*}
c_3:=\lim_{\ep\to 0}\frac{c_{\ep,0}-\frac{4A_0A_1}{R}}{\ep}\,.
\end{align*}
A straightforward computation using Equations~\eqref{formulaphi1} and~\eqref{cF} shows that
$$
\cF(\ep,0)=\ka-R^2c_3\ep + O(\ep^2)\,.
$$
We claim that $c_3=0$. Indeed, noticing that the previous results imply that
\begin{equation}\label{ce0}
2\pi\ep Rc_{\ep,0}=\int_0^{2\pi}|\nabla\phi_{\ep,0}(1,\te)|^2\cos \te \,d\te=8\pi\ep A_0A_1+2\pi \ep^2Rc_3+O(\ep^3)\,,
\end{equation}
to compute $c_3$ it is enough to obtain the $\ep^2$-term of $|\nabla\phi_{\ep,0}(1,\te)|^2$. Recall that the function $\phi_{\ep,0}$ is the solution to the Equation~\eqref{eq} with Dirichlet boundary condition $\phi_{\ep,0}(1,\te)=0$. According to Proposition~\ref{P.phi0}, it is easy to check that the ($\ep$-dependent) function $\phi_2$ defined as
$$
\phi_{\ep,0}=:A_0(\rho^2-1)+\ep A_1(\rho^2-1)\rho\cos\te+\ep^2\phi_2\,,
$$
satisfies the boundary value problem
$$
\Delta \phi_2=A_2+A_3x^2+A_4y^2+O(\ep)\,, \qquad \phi_2(1,\te)=0\,,
$$
for some explicit constants $A_2,A_3,A_4$ (depending on $R$ but not
on $\ep$) that are not be relevant for our purposes. The solution to
this problem is therefore of the form
$$
\phi_2=\frac{A_2}{4}(\rho^2-1)+\frac{A_3+A_4}{32}(\rho^4-1)+\frac{A_3-A_4}{24}\rho^2(\rho^2-1)\cos 2\te+O(\ep)\,.
$$
Using again Equation~\eqref{formulaphi1} we obtain that
$$
|\nabla \phi_{\ep,0}(1,\te)|^2=4A_0^2+8\ep A_0A_1\cos\te+4\ep^2[A_1^2\cos^2\te+A_0\nabla\phi_2(1,\te)\cdot e_\rho] +O(\ep^3)\,,
$$
where the scalar product $\nabla\phi_2(1,\te)\cdot e_\rho$ is given by
$$
\nabla\phi_2(1,\te)\cdot e_\rho=\frac{4A_2+A_3+A_4}{8}+\frac{A_3-A_4}{12}\cos 2\te+O(\ep)\,.
$$
It is then immediate that the $\ep^2$-term of $|\nabla\phi_{\ep,0}(1,\te)|^2$ does not contribute to the integral in Equation~\eqref{ce0}, thus proving that $c_3=0$ as claimed.
\end{proof}

It stems from Proposition~\ref{P.ceB} that the function
\begin{equation}\label{cG}
\cG(\ep,B):= \frac1\ep\left[\cF(\ep,B)-\ka \right]
\end{equation}
can be defined at $\ep=0$ by continuity, so that $\cG(0,0)=0$, resulting in a map defined
for all $|\ep|<\ep_0$, where $\ep_0$ is some positive constant. A more
convenient way of looking at this map, however, is by restricting our
attention to those variations of
the domain that are even and orthogonal to~$\cos\te$. Hence, let us
now define, for each non-integer $s>2$, the space
\begin{align*}
X_s :=\left\{ f\in C^s(\TT): f(\te)= f(-\te)\,,\; \langle f,\cos\te\rangle=0\right\}\,,
\end{align*}
and its ball of radius~$1$,
\begin{align*}
X_{s}^1 :=\left\{ f\in C^s(\TT): \|f\|_{C^s}<1,\; f(\te)= f(-\te)\,,\; \langle f,\cos\te\rangle=0\right\}\,.
\end{align*}
As $\langle \cF(\ep, B),\cos\te\rangle=0$ by the definition
of~$\ceB$, and $\phieB$ is an even function if $B$ is (cf. Proposition~\ref{P.existence}), our previous results then immediately imply the following:

\begin{proposition}\label{P.cG}
  Given any $R>0$, there is some $\ep_0>0$ such that the formula~\eqref{cG}
  defines a map
  \[
    \cG: (-\ep_0,\ep_0)\times X_{s+1}^1\to X_s\,.
  \]
\end{proposition}

In the following theorem we derive the key property of the map~$\cG$:
as its domain consists of the even functions orthogonal
to~$\cos\te$, we can show that its derivative with respect to~$B$ at
certain points is an invertible map:

\begin{theorem}\label{T.IFT}
For any $R>0$ such that $aR^2-3b\neq 0$, the Fr\'echet derivative
  \[
D_B\cG(0,0): X_{s+1}\to X_s
  \]
  is one-to-one.
\end{theorem}

\begin{proof}
  It follows from the definition of~$\cF$ (Equation~\eqref{cF}) and of~$\dphieBB$ that
  \begin{multline*}
D_B\cF(\ep,0)\dB= (2\nabla\phie \cdot \nabla\dphie+ \ep
\dB\, \pd_\rho|\nabla\phie|^2)|_{\rho=1} \\
-\cCe(R+\ep\cos\te)^2 -2\ce
\ep^2(R+\ep\cos\te) \dB\cos\te\,,
  \end{multline*}
  where the constant $\cCe$ is given by the derivative
  \[
\cCe:= \left.\frac\pd{\pd t}\right|_{t=0} c_{\ep,  t\dB}\,.
  \]

  We readily obtain from formulas~\eqref{formulaphi1}
  and~\eqref{formulas2} that
  \begin{multline}
    (2\nabla\phie \cdot \nabla\dphie+\ep \dB \pd_\rho
                              |\nabla\phie|^2)|_{\rho=1} =
                              8\ep A_0^2(\dB-\La\dB) \\
    +8\ep^2A_0
                              \left[4 A_1\, \dB\cos\te
                              -A_1\cos\te\,\La\dB -
                              A_1\La(\dB\cos\te)-\frac{A_0}R
                              T'\dB\right] +O(\ep^3)\,.\label{rollo1}
  \end{multline}
Since $\cCe$ is obviously of order $O(\ep)$, cf. Proposition~\ref{P.ceB}, it suffices to employ the leading order terms of this expression to
arrive at
\[
D_B\cF(\ep,0)\dB= 8\ep A_0^2(\dB-\La\dB)-\cCe R^2+O(\ep^2) \,.
\]

  Hence, in order
  to compute this derivative modulo an error of order~$\ep^2$ we only
  need to derive asymptotics for~$\cCe$. To do so, we write
  \[
c_{\ep,  t\dB}=c_1/c_2
  \]
  as in~\eqref{cj} (where now $B:= t\dB$) and compute
  \[
\cC_j:= \left.\frac\pd{\pd t}\right|_{t=0} c_j\,.
\]
Notice that, as we showed in the proof of
Proposition~\ref{P.ceB} that $c_j=O(\ep)$, we will need to compute $\cC_j$ to order
$O(\ep^2)$.

Let us start with $\cC_2$.  Since $c_2:= \langle [R+\ep (1+t\ep \dB)\cos\te]^2,\cos\te \rangle$,
it is immediate that
\begin{align*}
\cC_2=2\ep^2R\langle\dB,\cos^2\te\rangle + O(\ep^3)\,.
\end{align*}
To compute $\cC_1$, we again employ the
formula~\eqref{rollo1}, now to second order:
\begin{align*}
\cC_1&= \int_0^{2\pi} \left.\left( 2\nabla\phie \cdot \nabla\dphie
       +\ep\dB \pd_\rho
       |\nabla\phie|^2\right)\right|_{\rho=1}\cos\te\, d\te\\
     &=8\ep A_0^2 \langle \dB-\La\dB,\cos\te\rangle +8\ep^2A_0
                              \bigg[4 A_1 \langle\dB,\cos^2\te\rangle\\
                           &\qquad \qquad -A_1 \langle\La\dB,\cos^2\te
                             \rangle -
                              A_1\langle\La(\dB\cos\te),\cos\te\rangle-\frac{A_0}R
                              \langle T'\dB,\cos\te\rangle \bigg] +O(\ep^3)\\
  &= 8\ep^2A_0
                              \bigg(3 A_1 \langle\dB,\cos^2\te\rangle
                           -A_1 \langle\La\dB,\cos^2\te \rangle -\frac{A_0}R
                              \langle T'\dB,\cos\te\rangle \bigg) +O(\ep^3)\,.
\end{align*}
Here we have used that $\La$ is self-adjoint and that
$\La(\cos\te)=\cos\te$.

Now we need to compute the scalar products appearing in the two
previous formulas in terms of the Fourier coefficients~$\dB_n$ (note that $\dB_{-n}= \dB_n$ because~$\dB$ is even):
\begin{align*}
  \langle\dB,\cos^2\te\rangle&= \frac14\langle\dB,e^{2i\te}+e^{-2i\te}+2\rangle=\pi (\dB_2+\dB_0)\,,\\
  \langle\La\dB,\cos^2\te \rangle&=
                                   \frac14\langle\dB,\La(e^{2i\te}+e^{-2i\te}+2)\rangle=
                                   2\pi \dB_2\,,\\
\langle  T'\dB,\cos\te\rangle&= \frac14\sum_{n\in\ZZ\backslash\{0\}} \left\langle
                        \dB_n \, e^{i[n-\sgn(n)]\te} ,
                        {e^{i\te}+e^{-i\te}}\right\rangle = \pi\dB_2\,.
\end{align*}
Using the formulas for~$c_1$ and $c_2$ derived in the proof of
Proposition~\ref{P.ceB}, this immediately yields
\begin{align*}
  \cCe&=\frac{\cC_1}{c_2}-\ceB \frac{\cC_2}{c_2}\\
  &= -\ep\frac{8A_0^2}{R^2}\left(
    \frac12\dB_2-\frac{A_1R}{A_0}\dB_0\right) + O(\ep^2)\,,
\end{align*}
which results in
\[
D_B\cF(\ep,0)\dB= 8\ep A_0^2\left(
  \dB-\La\dB+\frac12\dB_2-\frac{A_1R}{A_0}\dB_0\right)+ O(\ep^2)\,.
\]

We are now ready to analyze the differential of~$\cG$, which we have
shown to be given by the formula
\[
D_B\cG(0,0)\dB= 8 A_0^2\left(
  \dB-\La\dB+\frac12\dB_2-\frac{A_1R}{A_0}\dB_0\right)\,,
\]
understood as a map $X_{s+1}\to X_s$. We recall that, as $\dB$ is an
even function ortogonal to $\cos\te$, the Fourier
series~$\dB=\sum_{n\in\ZZ} \dB_n e^{in\te}$ can be equivalently
written as
\[
\dB(\te)= \dB_0 + 2\sum_{n=2}^\infty \dB_n\cos n\te\,.
\]
Therefore, the action of the linear elliptic operator $D_B\cG(0,0)$ is
given by
\begin{align*}
D_B\cG(0,0)\dB= 8 A_0^2\left[
  \left(1-\frac{A_1R}{A_0}\right)\dB_0+\frac12\dB_2 -2\sum_{n=2}^\infty (n-1)\dB_n\cos n\te \right]\,.
\end{align*}
Note that $A_1R\neq A_0$ for all $a,b,R>0$ such that $aR^2-3b\neq 0$ because
\[
A_0-A_1R=\frac{3b-aR^2}{16}\,.
\]
This implies that the kernel of the map $D_B\cG(0,0): X_{s+1}\to
X_s$ is trivial, and that its range is the whole space~$X_s$, as claimed.
\end{proof}

In the following corollary we show that, by the implicit function theorem for Banach spaces,
Theorem~\ref{T.IFT} yields the existence of solutions to the
overdetermined boundary value
problem~\eqref{elliptic}--\eqref{Neumann} for all small enough~$\ep$
and all~$R$ such that $aR^2-3b>0$. In turn, these define piecewise
smooth stationary Euler flows of compact support via
Lemma~\ref{L.weak}, thereby completing the proof of the main result of
the paper (Theorem~\ref{T.main}).  Recall that the constant $F_R$ appears in the definition of the function $F$, cf.~Equation~\eqref{FH}.

\begin{corollary}\label{C.IFT}
  Fix any $R>0$ such that $aR^2-3b>0$. Then, for any small enough~$\ep$ there is a unique $B\in X_{s+1}$ in a
  $C^{s+1}$~neighborhood of~$0$ such that $\psi:= \ep^2\phieB$
  satisfies Equation~\eqref{elliptic} in $\Om_{R,\ep}:=\Om_{\ep B}$ and the
  overdetermined boundary conditions~\eqref{Dirichlet}-\eqref{Neumann}
  with $F_R:=-\ka>0$ and $c=\ep^2 c_{\ep,B}$.
  \end{corollary}

  \begin{proof}
Since $\cG(0,0)=0$, in view of Theorem~\ref{T.IFT}, the implicit function theorem
guarantees that if $|\ep|$ is small enough, there is a unique
function~$B$ in a small neighborhood of~$0$ in~$X_{s+1}^1$ such that
\[
\cG(\ep,B)=0\,.
\]
This is equivalent to saying that
\[
|\nabla\psi|^2-\ep^2\ceB r^2-\ep^2 \ka  =0
\]
on~$\pd\Om_{\ep B}$, with $\psi:=\ep^2\phieB$. The assumption that $F^2(0)=\ep^2F_R =-\ep^2\ka$ then
ensures that we have a solution to the overdetermined boundary
problem~\eqref{elliptic}--\eqref{Neumann}, as claimed. Observe that the condition $aR^2-3b>0$ implies that
$$
\ka=\frac{(aR^2+b)(3b-aR^2)}{16}<0
$$
and hence $F_R>0$. Accordingly, the function $F(\psi)$ is well defined:
$$
F(\psi)=\Big(\ep^2F_R-2b\psi+O(\psi^2)\Big)^{1/2}
$$
because $\psi=O(\ep^2)$ and $\psi<0$ in $\Om_{\ep B}$ (cf. Proposition~\ref{P.existence}).
  \end{proof}


\section{Different choices for the functions~$F$ and~$H$}\label{S.final}

As we mentioned in the Introduction, for the sake of concreteness we
have chosen the functions~$H$ and~$F$ as described in
Theorem~\ref{T.main}. However, the method introduced in this paper is
flexible enough to construct compactly supported stationary Euler
flows with other choices for the functions~$H$ and~$F$. To illustrate
this additional flexibility, in this section we show how a
straightforward modification of the previous computations allows us to prove the following:

\begin{theorem}
Take any non-integer $s>2$ and any
functions $\tF,H\in C^s((-1,0])$ with
\[
\tF(0)=\tF'(0)=0\,,\qquad H'(0)> 0\,.
\]
Then the following statements hold:
\begin{enumerate}
\item For each small enough $\ep>0$ and any $R>0$,
  there exists a nontrivial, piecewise~$C^s$, axisymmetric stationary
  Euler flow of compact support~$u$ of the form described in Lemma~\ref{L.weak} for
  a suitable $C^{s+1}$~planar domain~$\Om_{R,\ep}$.

\item The boundary of~$\Om_{R,\ep}$  is a small
  deformation of a disk of radius~$\ep$, given by an equation of the form $z^2+(r-R)^2-\ep^2 = O(\ep^3)$.

\item The functions that define the solution are
    \begin{equation*}
F(\psi) := \ep F_R+ \tF(\psi)
\end{equation*}
and $H(\psi)$, where $F_R$ is the positive constant
\begin{equation}\label{FR}
F_R:=\frac{R^2H'(0)}{4}\,.
\end{equation}
\item The function $\psi$ is of class $C^{s+1}$ in $\Om_{R,\ep}$ up to the boundary, and has the form
$$
\psi=\frac14 H'(0)R^2\Big[(r-R)^2+z^2-\ep^2\Big]+O(\ep^3)\,.
$$
Moreover, $F\circ \psi>0$ and $H\circ\psi$ are of class $C^s$
in $\Om_{R,\ep}$. In particular, the vorticity is of class $C^{s-1}$ up the boundary.
    \end{enumerate}
\end{theorem}

\begin{proof}
Indeed, using the same notation as in Section~\ref{S.Dirichlet}, and noticing that
$$(F^2)'(\psi)=\ep O(\psi)\,,$$
Equation~\eqref{eq} takes the form
\begin{equation*}
\De\phi -\frac\ep{R+\ep x}\pd_x \phi= aR^2+ 2aR\ep x + O(\ep^2)\,,
\end{equation*}
where we have defined the constant $a:=H'(0)$. Notice that this is exactly the same as Equation~\eqref{eq} with $b=0$. Repeating all the arguments in Sections~\ref{S.Dirichlet}--\ref{S.Neumann}, we obtain the same equations and results as in these sections with $b=0$. In particular, the constant $\ka$ in Proposition~\ref{P.ceB} is given by
\[
\ka=-\frac{a^2R^4}{16}<0\,,
\]
$\cF(\ep,0)=\ka+O(\ep^2)$, and the invertibility condition in Theorem~\ref{T.IFT} is simply $a\neq0$. The Neumann boundary condition is then satisfied taking $F(0)=\ep F_R$, with $F_R$ as in Equation~\eqref{FR}. Notice that $F(\psi)=\ep F_R+O(\psi^2)=\ep F_R+O(\ep^4)>0$ in $\Om_{R,\ep}$.
\end{proof}

\section*{Acknowledgements}

M.D.-V.\ is supported by the grants MTM2016-75897-P (AEI/FEDER)
and ED431F 2017/03 (Xunta de Galicia), and by the Ram\'on y Cajal program of the Spanish Ministry
of Science. A.E.\ is supported by the ERC Starting
Grant~633152. D.P.-S.\ is supported by the grants MTM2016-76702-P
(MINECO/FEDER) and Europa Excelencia EUR2019-103821 (MCIU). This work
is supported in part by the ICMAT--Severo Ochoa grant SEV-2015-0554
and the CSIC grant 20205CEX001.

\bibliographystyle{amsplain}

\end{document}